\documentclass[jsl,reqno]{asl}

\usepackage{verbatim}
\usepackage[utf8]{inputenc}
\usepackage{multicol}

\usepackage{tikz}
\usetikzlibrary{matrix,arrows}

\makeatletter
\newtheorem*{rep@theorem}{\rep@title}
\newcommand{\newreptheorem}[2]{%
\newenvironment{rep#1}[1]{%
 \def\rep@title{#2 \ref{##1}}%
 \begin{rep@theorem}}%
 {\end{rep@theorem}}}
\makeatother

\theoremstyle{plain}
\newtheorem{theorem}{Theorem}
\newreptheorem{theorem}{Theorem}
\newtheorem{proposition}[theorem]{Proposition}
\newtheorem{lemma}[theorem]{Lemma}
\newreptheorem{lemma}{Lemma}

\theoremstyle{definition}
\newtheorem{definition}[theorem]{Definition}

\newcommand{\proves}{\vdash}

\renewcommand{\theta}{\vartheta}
\renewcommand{\phi}{\varphi}

\newcommand{\V}{\mathcal{V}}

\newcommand{\LTLI}{$(\diam,\X,\I)$}
\newcommand{\SF}[1]{\mathbf{S1S}(#1)} 

\newcommand{\diam}{\lozenge}
\newcommand{\N}{\mathbf{X}}
\newcommand{\X}{\N}
\newcommand{\I}{I}

\newcommand{\bvee}{\vee}
\newcommand{\bneg}{\neg}
\newcommand{\bbot}{\bot}
\newcommand{\btop}{\top}

\numberwithin{theorem}{section}

\newif\ifshowproofs
\showproofsfalse

\newcommand{\auxproof}[1]{
\ifshowproofs
{\small \textit{Proof.}
{#1}
\qed}
\fi}

\renewcommand{\labelenumi}{(\roman{enumi})}

\title[Model-theoretic characterization of MSO on infinite words]{A model-theoretic characterization of \\ monadic second order logic on infinite words}
\author[Ghilardi and van Gool]{Silvio Ghilardi and Samuel J. van Gool}

\date{\today}		

\begin{document}

\begin{abstract}
Monadic second order logic and linear temporal logic are two logical formalisms that can be used to describe classes of infinite words, i.e., first-order models based on the natural numbers with order, successor, and finitely many unary predicate symbols.

Monadic second order logic over infinite words (S1S) can alternatively be described as a first-order logic interpreted in $\mathcal{P}(\omega)$, the power set Boolean algebra of the natural numbers, equipped with modal operators for `initial', `next' and `future' states.
We prove that the first-order theory of this structure
is the model companion of a class of algebras corresponding to a version of linear temporal logic (LTL) without until.

The proof makes crucial use of two classical, non-trivial results from the literature, namely the completeness of LTL with respect to the natural numbers, and the correspondence between S1S-formulas and B\"uchi automata. \\
\end{abstract}
\maketitle

\section{Introduction}

Monadic second order logic over the natural numbers with successor operation is a rather expressive, but still decidable formalism. The decision result, originally due to B\"uchi~\cite{Buc1962}, makes use of a conversion between logic and automata. The key idea is to view interpretations of unary predicates over natural numbers as infinite words over a suitable alphabet: one can associate an automaton with a formula, and vice versa, in such a way that, roughly speaking, the automaton accepts exactly those infinite words that, viewed as interpretations of second-order variables, satisfy the formula. 
Converting a formula $\phi$ into an automaton $A_\phi$ and then the automaton $A_\phi$ again into a formula,  one does not get back the same formula $\phi$, but a formula $\phi'$ which is equivalent to $\phi$ in the intended
model of the natural numbers.
Loosely speaking, one may view $\phi'$ as a kind of `normalization' of $\phi$. Morally, $\phi'$ is an existential formula; although this is not formally true,  the `existential nature' of $\phi'$ is rather evident. Nevertheless, in order to make $\phi'$ into an actual existential formula, an enlargement of the language is needed.

In this paper we make the above observations precise, which allows us to fit monadic second order logic into the framework of modern model-theoretic algebra (Robinson's legacy~\cite{Robinson}), using the concepts of model completeness and model companions. We briefly recall these concepts now; precise definitions are given in Section~\ref{sec:statement}.
Theories $T^*$ with the property that every formula is equivalent modulo $T^*$ to an existential (equivalently: to a universal) formula  are called \emph{model complete}; model complete 
theories $T^*$ can also be characterized semantically by requiring that all embeddings between models of  $T^*$ are elementary. Algebraically closed fields and real closed fields are classical examples of model complete theories (see, e.g., \cite{Whe1976} for many further examples). 
Moreover, as the previous paragraph indicates, monadic second order logic, viewed as the first-order theory of  $\mathcal{P}(\omega)$, is close to being model complete; we will see later in this paper (Section~\ref{sec:modelcomplete}) that it actually becomes model complete in an enriched language.
If $T$ is the set of universal consequences of a model complete theory $T^*$, then $T^*$ axiomatizes the class of structures which, as models of $T$, 
are \emph{existentially closed}: for the case of algebraically closed fields, $T$ is the theory of integral domains, for the case of real closed fields, $T$ is the theory of formally real fields, etc. 
In this situation, $T^*$ is called a \emph{model companion} of $T$. 

In this paper, we exhibit a natural enrichment of the language of monadic second order logic, viewed as the first-order theory of $\mathcal{P}(\omega)$, so that the theory becomes model complete. Moreover, we axiomatize the underlying universal fragment, i.e., the universal theory $T$ of which monadic second order logic over natural numbers is the model companion.
Our idea is to use \emph{modal and temporal logic}  to attain this goal. More precisely: we will use a variant of the modal logic LTL (`linear temporal logic'~\cite{Pnueli}) 
with an actual instant of time (i.e. with an atom for zero) and without until; i.e., we have a basic modal logic with a 
reachability operator and one additional constant.\footnote{
The reason for leaving out the Until operator is that we opted for as small a language as possible; Until is not needed for our purpose of expressing the B\"uchi acceptance condition. 
In fact, any operator definable in monadic second order logic over the natural numbers could be added to our $T$ 
(together with its axiomatization) without affecting our results.
} 
We introduce a universal class of algebras algebraizing this logic and we show that this class fits our purposes. 
Indeed, this class is axiomatized by a theory $T$ admitting
a model companion $T^*$ which is exactly monadic second order logic over the natural numbers (Theorem~\ref{thm:main}).  
In order to establish this result, we shall make use, besides automata techniques, of standard modal logic machinery, in particular, duality and filtrations~\cite{Esa1974,Gol1989,Gol1992}.

The paper is organized in the following way: in Section~\ref{sec:statement} we state our result, in the (rather technical) Section~\ref{sec:companion} we give the part of the proof requiring modal logic techniques (completeness of the axiomatization) and in 
Section~\ref{sec:modelcomplete} the part of the proof requiring automata theory ingredients (model completeness of the theory in the enriched language). Section~\ref{sec:conclusions} concludes and indicates directions for future work.

\section{Definitions and statement of the main theorem}\label{sec:statement}
In the following definition, we introduce the relevant class of algebras.
\begin{definition}\label{dfn:LTLIalgebra}
A \emph{\LTLI-algebra} is a tuple $(A,\vee,\neg,\bot,\diam,\X,I)$, where \\$(A, \vee, \neg, \bot)$ is a Boolean algebra, $\diam$ is a unary normal modal operator on $A$ (i.e., a $(\vee,\bot)$-semilattice endomorphism of $A$), $\X$ is a Boolean endomorphism of $A$, $I$ is an element of $A \setminus \{\bot\}$, and, for any $a \in A$, the following conditions hold:
\begin{multicols}{2}
\begin{enumerate}
\item $\diam a = a \vee \X \diam a$,
\item if $\X a \leq a$ then $\diam a \leq a$,
\item if $a \neq \bot$ then $I \leq \diam a$.
\item $\X I = \bot$.
\end{enumerate}
\end{multicols}
\end{definition}
The class of \LTLI-algebras (which is a \emph{universal} class) algebraizes a version of linear temporal logic without the until connective and with a constant $I$ for `initial element'. The structure of \LTLI-algebras will be studied in detail in Section~\ref{sec:companion}.

An important example of a \LTLI-algebra is the \emph{power set algebra of the natural numbers}, $\mathcal{P}(\omega)$, equipped with the usual Boolean operations, and the operations $\diam S := \{ n \in \omega \ | \ n \leq s \text{ for some } s \in S\}$, $\X S := \{ n \in \omega \ | \ n + 1 \in S\}$, and $I := \{0\}$. We will see in Section~\ref{sec:modelcomplete} that first-order formulas in the signature of \LTLI-algebras, interpreted in $\mathcal{P}(\omega)$, are interdefinable with formulas in the system S1S, monadic second order logic over the natural numbers with order and successor relations.

If $T$ and $T^*$ are first-order theories in the same signature, recall that $T^*$ is called a \emph{model companion} of $T$ if (i) the theories $T$ and $T^*$ have the same universal consequences (i.e., $T^*$ is a \emph{companion} or \emph{co-theory} of $T$) and (ii) any first-order formula is equivalent over $T^*$ to an existential formula (i.e., $T^*$ is {\it model complete}). The model companion of $T$ is unique if it exists, and in this case it is the theory of the existentially closed $T$-models \cite{Whe1976}. 

Our main theorem is the following.
\begin{theorem}\label{thm:main}
The first-order theory of the \LTLI-algebra $\mathcal{P}(\omega)$ is the model companion of the first-order theory of \LTLI-algebras.
\end{theorem}
In the rest of this paper, we denote by $T$ the first-order theory of \LTLI-algebras, and by $T^*$ the first-order theory of the \LTLI-algebra $\mathcal{P}(\omega)$. 
In Section~\ref{sec:companion}, we prove that $T^*$ is a companion of $T$. In Section~\ref{sec:modelcomplete}, we prove that $T^*$ is model complete. Together, these two sections prove Theorem~\ref{thm:main}.

\section{$T^*$ is a companion of $T$}\label{sec:companion}
In this section, we prove that $T^*$ is a companion of $T$, i.e., that $T$ and $T^*$ have the same universal consequences.\\

{\bf Notation.} Throughout this section, we denote by $\mathcal{L} = \{\bvee,\bneg,\bbot,\diam,\X,\I\}$ the first-order signature of \LTLI-algebras. We employ the usual abbreviations $a \wedge b := \neg(\neg a \vee \neg b)$, $\Box a := \neg \diam \neg a$, $a \to b := \neg a \vee b$, and $a \leftrightarrow b := (a \to b) \wedge (b \to a)$.\\

We begin with a useful observation: in the theory $T$, we can reduce quantifier-free $\mathcal{L}$-formulas to single equations.
\begin{lemma}\label{lem:equation}
For any quantifier-free $\mathcal{L}$-formula $\phi$, there exists an $\mathcal{L}$-term $t$ such that $\phi$ is equivalent to $t = \btop$ in $T$.
\end{lemma}
\begin{proof}
Observe that, for any element $w$ in a \LTLI-algebra, we have
\begin{align}
\tag{*}
 w \neq \btop \text{ if, and only if, } I \leq \diam \neg w \text{ if, and only if, } I \to \diam \neg w = \btop.
\end{align}

Let $\phi$ be a quantifier-free formula. We may assume that $\phi$ is a disjunction of formulas $\psi_1, \dots, \psi_n$, where each $\psi_j$ is a conjunction of literals, i.e., has the form
\[ r^j_1 \neq s^j_1 \text{ and } \dots \text{ and } r^j_k \neq s^j_k \text{ and } u^j_1 = v^j_1 \text{ and } \dots \text{ and } u^j_\ell = v^j_\ell,\]
where the $r^j_i$, $s^j_i$, $u^j_i$, $v^j_i$ are $\mathcal{L}$-terms. Define the $\mathcal{L}$-terms
\[t_j := \bigwedge_{i=1}^k (I \to \diam \neg(r^j_i \leftrightarrow s^j_i)) \wedge \bigwedge_{i=1}^\ell (u^j_i \leftrightarrow v^j_i),\quad t := I \to \bigvee_{j=1}^n \Box t_j.\]
Using (*) and basic facts about Boolean algebras, we obtain that $\psi_j$ is equivalent in $T$ to $t_j = \top$, so $\phi$ is equivalent in $T$ to the disjunction ($t_1 = \top$ or ... or $t_n = \top$). This disjunction, in turn, is equivalent in $T$ to $t = \top$, as can be proved easily using (*) again, together with the axioms $\Box \top = \top$ and $I \neq \bot$.
\auxproof{ 
Indeed, if $t_j = \top$ for some $1 \leq j \leq n$, then $\Box t_j = \top$ so $\bigvee_{j=1}^n \Box t_j = \top$. Conversely, if $t_j \neq \top$ for each $1 \leq j \leq n$, then by (*) $I \leq \diam \neg t_j$ for each $j$, so $I \leq \bigwedge_{j=1}^n \diam \neg t_j = \neg \left(\bigvee_{j=1}^n \Box t_j \right)$. Since $I \neq \bot$, we must have $I \not\leq \bigvee_{j=1}^n \Box t_j$.
}
\end{proof}

In order to prove that $T^*$ is a companion of $T$, since $T \subseteq T^*$, it suffices to prove that any universal formula that is valid in $\mathcal{P}(\omega)$ is valid in any \LTLI-algebra. Therefore, by Lemma~\ref{lem:equation}, it is enough to prove the following theorem.

\auxproof{
Since $T \subseteq T^*$, any quantifier-free formula that is provable from $T$ is provable from $T^*$. For the converse, it suffices, by Lemma~\ref{lem:equation}, to prove for any term $t$ that if $T^* \proves t = \btop$ then $T \proves t = \btop$. If $T^* \proves t = \btop$, then in particular $\mathcal{P}(\omega) \models t = \btop$. By Theorem~\ref{thm:compl}, $A \models t = \btop$ for any \LTLI-algebra (i.e., $T$-model) $A$. By first-order completeness, $T \proves t = \btop$.
}
\begin{theorem}\label{thm:completeness}
If $t$ is an $\mathcal{L}$-term and $\mathcal{P}(\omega) \models t = \btop$, then, for any \LTLI-algebra $A$, $A \models t = \btop$.
\end{theorem}

In the rest of this section, we prove Theorem~\ref{thm:completeness}. The proof we give here is an adaptation to our setting of the standard completeness theorem for LTL with respect to $\omega$; in particular, it is directly inspired by the proof given in \cite[Ch.~9]{Gol1992}. Our proof of Theorem~\ref{thm:completeness} is structured as follows. We will first show, in Theorem~\ref{thm:duality}, that \LTLI-algebras can be represented as dual algebras of \emph{\LTLI-spaces} (Definition~\ref{dfn:LTLIspace}), by means of an adaptation of the standard (Stone-)J\'onsson-Tarski representation theorem. We then combine this theorem with an adaptation of a filtration argument for LTL (Lemma~\ref{lem:filtration}) to prove Theorem~\ref{thm:completeness}.

Below, we use the following notational conventions for a preorder (= set endowed with a reflexive and transitive relation) $(X, \leq)$. 
We define $x < y$ as ($x\leq y$ and $x\neq y$); for $Y \subseteq X$, we write ${\uparrow} Y := \{ x \in X \mid y \leq x \text{ for some } y \in Y\}$, and similarly ${\downarrow} Y$; finally, for $x \in X$, we write ${\uparrow} x$ and ${\downarrow} x$ as shorthand for ${\uparrow} \{x\}$ and ${\downarrow} \{x\}$, respectively.

We will now formulate the duality between algebras and general frames (viewed as Stone spaces) that we need here. Our exposition will be brief, as we are only using standard modal logic machinery. We refer to, e.g., \cite[Ch.~5]{BRV2001}, for more details on duality for modal algebras. To motivate the definition of \LTLI-spaces, observe that the $\diam$-reduct of a \LTLI-algebra is an S4-algebra. Through J\'onsson-Tarski representation, S4-algebras correspond to Boolean spaces equipped with a \emph{topological preorder}, i.e., a preorder such that ${\uparrow} x$ is closed for any point $x$ and ${\downarrow} K$ is clopen for any clopen set $K$ (see, e.g., \cite[Rem.~1 in Sec.~2.6]{Geh2014}). The additional structure and properties of \LTLI-algebras now correspond to additional structure and properties of these preordered Boolean spaces, as follows.
\begin{definition}\label{dfn:LTLIspace}
We define a \emph{\LTLI-space}\footnote{Note that our definition of \LTLI-spaces makes crucial use of the second-order structure (topology) on the underlying Kripke frames. This is necessarily so: the class of \LTLI-algebras is not canonical, so it can not be dual to an elementary class of Kripke frames, by a theorem of Fine \cite{Fin1975}.} to be a tuple $(X,\leq,f,x_0)$, where $X$ is a Boolean topological space, $\leq$ is a topological preorder on $X$, $f : X \to X$ is a continuous function, $x_0 \in X$ is a point such that $\{x_0\}$ is clopen, and, for any $x,y  \in X$ and clopen $K \subseteq X$:

\begin{multicols}{2}
\begin{enumerate}
\item $x \leq f(x)$, and if $x < y$ then $f(x) \leq y$,
\item if $f(K) \subseteq K$ then ${\uparrow} K \subseteq K$,
\item $x_0 \leq x$,
\item $f(x) \neq x_0$.
\end{enumerate}
\end{multicols}

The \emph{dual algebra} of a \LTLI-space $(X,\leq,f,x_0)$ is defined to be the tuple $(A,\diam,\X,I)$, where $A$ is the Boolean algebra of clopen subsets of $X$, and for any $K \in A$, $\diam K := {\downarrow} K$, $\X K := f^{-1}(K)$, and $I := \{x_0\}$.
\end{definition}

We now prove our representation theorem for \LTLI-algebras.

\begin{theorem}\label{thm:duality}
The class of \LTLI-algebras coincides with the class of algebras that are isomorphic to the dual algebra of a \LTLI-space.
\end{theorem}
\begin{proof}
Note that, \emph{in any \LTLI-algebra, $\I$ is an atom}: let $\bot < a \leq \I$ be arbitrary. Then $I \leq \diam a$
by Def.~\ref{dfn:LTLIalgebra}(iii). Also, $\diam \I \leq \I$ by Def.~\ref{dfn:LTLIalgebra}(ii)~and~(iv). Using these facts, Def.~\ref{dfn:LTLIalgebra}(i) and the monotonicity of $\X$ and $\diam$, we obtain that $I \leq a \vee \X \diam a \leq a \vee \X \diam I \leq a \vee \X I$,
i.e., $I\leq a$ by Def.~\ref{dfn:LTLIalgebra}(iv), as required.

Now, by the J\'onsson-Tarski representation theorem (see, e.g., \cite[Section 5.3]{BRV2001}), the class of algebras $(A,\diam,\X,\I)$ where $(A,\diam)$ is an S4-algebra, $\X$ is an endomorphism, and $\I$ is an atom of $A$, coincides with the class of algebras that are isomorphic to the dual algebra of a tuple of the form $(X,\leq,f,x_0)$, where $X$ is a Boolean space, $\leq$ is a topological preorder on $X$, $f$ is a continuous function on $X$ and $x_0 \in X$ is such that $\{x_0\}$ is clopen.

It remains to prove that $(X,f,\leq,x_0)$ validates (i)-(iv) in the definition of \LTLI-space if, and only if, its dual algebra $A$ is a \LTLI-algebra. This follows from the following claim.

{\bf Claim.} The following equivalences hold.
\begin{enumerate}
\item[(a)] $x \leq f(x)$ for all $x \in X$ $\iff$ $\N \diam a \leq \diam a$ for all $a \in A$.
\item[(b)] if $x < y$ then $f(x) \leq y$ for all $x, y \in X$ $\iff$ $\diam a \leq a \vee \N \diam a$ for all $a \in A$.
\item[(c)] $(X, f, \leq)$ validates (ii) in Definition~\ref{dfn:LTLIspace} $\iff$ $A$ validates (ii) in Definiton~\ref{dfn:LTLIalgebra}.
\item[(d)] $x_0 \leq x$ for all $x \in X$ $\iff$ $\I \leq \diam a$ for all $a \in A \setminus \{\bot\}$.
\item[(e)] $f(x) \neq x_0$ for all $x \in X$ $\iff$ $\X \I = \bot$.
\end{enumerate}

{\it Proof of Claim.}
(a) ($\Rightarrow$) Let $a \in A$ be arbitrary, and suppose that $x \in \N \diam a$. Then $f(x) \in \diam a$, and $x \leq f(x)$, so $x \in \diam a$.\\
($\Leftarrow$) Suppose that $x \nleq f(x)$ for some $x \in X$. Since ${\uparrow} x$ is closed, there exists a clopen set $a \in A$ such that $f(x) \in a$ and ${\uparrow} x \cap a = \emptyset$. In particular, $f(x) \in \diam a$, so $x \in \N\diam a$, but $x \not\in \diam a$.

(b) ($\Rightarrow$) Let $a \in A$ be arbitrary, and suppose that $x \in \diam a$ but $x \not\in a$. Since $x \in \diam a$, pick $y \geq x$ such that $y \in a$. Since $x \not\in a$, we have $x \neq y$, so $f(x) \leq y$. Hence, $x \in \N\diam a$.\\
($\Leftarrow$) Suppose that there exist $x, y \in X$ such that $x < y$, but $f(x) \nleq y$. Since ${\uparrow} f(x)$ is closed, pick $a_1 \in A$ such that $y \in a_1$ and ${\uparrow} f(x) \cap a_1 = \emptyset$. Since $x \neq y$, pick $a_2 \in A$ such that $y \in a_2$ and $x \not\in a_2$. Let $a := a_1 \cap a_2$. Note that $x \in \diam a$, since $x \leq y$ and $y \in a$. However, we have $x \not\in a$, and $x \not\in \N\diam a$, contrary to the assumption that $\diam a \leq a \vee \N \diam a$ for all $a \in A$.

(c) Note that $f(K) \subseteq K$ if, and only if, $f^{-1}(K^c) \subseteq K^c$, and that ${\uparrow} K \subseteq K$ if, and only if, ${\downarrow} (K^c) \subseteq K^c$. The stated equivalence now follows from the definitions of $\X$ and $\diam$.

(d) ($\Rightarrow$) If $a \in A \setminus \{\bot\}$, then there is $x \in a$. Since $x_0 \leq x$, it follows that $\I = \{x_0\} \leq \diam a$.\\
($\Leftarrow$) If $x_0 \nleq x$, pick $a \in A$ such that $x \in a$ and ${\uparrow} x_0 \cap a = \emptyset$. Then $a \neq \bot$, but $x_0 \not\in \diam a$, so $\I \nleq \diam a$.

(e) Clear from the definitions.
\end{proof}

We may use the representation theorem, Theorem~\ref{thm:duality}, to prove the following proposition, which will be used in the proof of Lemma~\ref{lem:filtration}.

\begin{proposition}\label{prop:3dum}
Any \LTLI-algebra validates the equations
\begin{enumerate}
\item[(Con)] $\Box(\Box a \to b) \vee \Box (\Box b \to a) = \btop$,
\item[(Dum)] $\Box(\Box(a \to \Box a) \to a) \wedge \diam \Box a \leq a$.
\end{enumerate}
Moreover, the preorder on any \LTLI-space is linear.
\end{proposition}
\begin{proof}
Let $A$ be a \LTLI-algebra dual to a \LTLI-space $(X,\leq,f,x_0)$. 

(Con) Let $a, b \in A$ be arbitrary, and write $K := \diam \neg(\Box a \to b) \wedge \diam \neg (\Box b \to a)$, the complement of $\Box(\Box a \to b) \vee \Box (\Box b \to a)$. We need to show that $K = \emptyset$. We prove first that $f(K) \subseteq K$. If $x \in K$, pick $y \in \neg(\Box a \to b) = \Box a \wedge \neg b$ and $z \in \neg(\Box b \to a) = \Box b \wedge \neg a$ with $y, z \geq x$. Then $y \neq x$, since $y \in \Box a$ but $z \not\in a$, while $x \leq z$. So $f(x) \leq y$. Similarly, $f(x) \leq z$, so $f(x) \in K$. By Def.~\ref{dfn:LTLIspace}(ii), we obtain that ${\uparrow} K \subseteq K$. Now, if we would also have that $K \neq \emptyset$, then there would exist $x \in K$ and $y \geq x$ such that $y \in \Box a \wedge \neg b$. But then $y \in \Box (\Box b \to a)$, so $y \not\in K$, contradicting that ${\uparrow} K \subseteq K$. Therefore, we must have $K = \emptyset$, as required.

(Dum) Let $a \in A$ be arbitrary. We will prove $\Box(\Box(a \to \Box a) \to a) \wedge \diam \Box a \leq \Box a$, from which (Dum) follows since $\Box a \leq a$. The method is the same as before. Let us write $K := \Box(\Box(a \to \Box a) \to a) \wedge \diam \Box a \wedge \neg \Box a$. We need to show that $K = \emptyset$. We prove first that $f(K) \subseteq K$. Let $x \in K$ be arbitrary. Pick $y \geq x$ such that $y \not\in a$, and pick $z \geq x$ such that $z \in \Box a$. Then $z \neq x$, so $f(x) \leq z$, and $f(x) \in \diam \Box a$. Also, since ${\uparrow} (\Box b) \subseteq \Box b$ holds for any $b$, and $x \leq f(x)$, we have $f(x) \in \Box (\Box(a \to \Box a) \to a)$. It remains to prove that $f(x) \in \neg \Box a$. If $y \neq x$, then $f(x) \leq y$, and we are done. Otherwise, we have that $y = x$, so $x \not\in a$. Since $x \in \Box(a \to \Box a) \to a$, pick $w \geq x$ such that $w \in a$ and $w \in \neg \Box a$. Then $w \neq x$, so $f(x) \leq w$, so $f(x) \in \neg \Box a$, as required. Again, by Def.~\ref{dfn:LTLIspace}(ii), we obtain ${\uparrow} K \subseteq K$, and if $K$ were non-empty, we would have $x \in K$ and $z \geq x$ such that $z \in \Box a$, but then $z \not\in K$, contradiction. So $K = \emptyset$, as required.

Finally, since (Con) holds in $A$, the preorder $\leq$ is connected, i.e., for any $z,x,y \in X$, if $z \leq x$ and $z \leq y$, then $x \leq y$ or $y \leq x$ (see, e.g., \cite[Ex.~4.3.3]{BRV2001}). Since the \LTLI-space $(X,\leq,f,x_0)$ has a minimum element, $x_0$, it follows that the preorder is linear.
\end{proof}

We now turn to the filtration argument. The following syntactic lemma, which allows us to rewrite terms into negation normal form, will be useful. We call an $\mathcal{L}$-term a \emph{literal} if it is either a variable, a constant ($\bbot$ or $\I$), a negated variable or a negated constant. An $\mathcal{L}$-term is in \emph{negation normal form} (NNF) if it is built up from literals by repeated applications of $\vee$, $\wedge$, $\diam$, $\Box$ and $\N$.
\begin{lemma}\label{lem:NNF}
For any $\mathcal{L}$-term $t$, there exists an $\mathcal{L}$-term $t'$ in negative normal form such that $T \proves t = t'$.
\end{lemma}
\begin{proof}
By induction on the complexity of the $\mathcal{L}$-term $t$, we can push any negation inwards, making use of the following $T$-provable equalities: $\neg(s_1 \vee s_2) = \neg s_1 \wedge \neg s_2$, $\neg(s_1 \wedge s_2) = \neg s_1 \vee \neg s_2$, $\neg \diam s = \Box \neg s$, $\neg \Box s = \diam \neg s$ and $\neg \N s = \N \neg s$.
\end{proof}

\begin{definition}
A finite set $\Gamma$ of NNF $\mathcal{L}$-terms is {\it filterable} if
\begin{enumerate}
\item $\Gamma$ contains $I$,
\item $\Gamma$ is closed under subterms,
\item whenever $\Gamma$ contains $\diam s$ for some term $s$, $\Gamma$ also contains $\N \diam s$,
\item whenever $\Gamma$ contains $\Box s$ for some term $s$, $\Gamma$ also contains $\N \Box s$.
\end{enumerate}
\end{definition}

\begin{lemma}\label{lem:filtration}
Let $\Gamma$ be a filterable set, and denote by $v_1,\dots,v_N$ the variables occurring in $\Gamma$. For any \LTLI-space $(X,f,\leq_X, x_0)$ with dual algebra $A$ and $\overline{a} \in A^N$, there exists $\overline{p} \in \mathcal{P}(\omega)^N$ with the following property: for any $x \in X$, there is $n_x \in \omega$ such that for all $s \in \Gamma$,
\begin{align}\label{eq:star}
\tag{$\star$}
\text{if } x \in s^A(\overline{a}) \text{ then } n_x \in s^{\mathcal{P}(\omega)}(\overline{p}).
\end{align}
\end{lemma}
\begin{proof}
Let $(X,f,\leq_X, x_0)$ be a \LTLI-space and let $\overline{a} \in A^N$. Throughout this proof, if $x \in X$ and $s$ is a term, we will write ``$x \in s$'' as shorthand for ``$x \in s^A(\overline{a})$''.

Let $\sim$ be the equivalence relation on $X$ defined by
\[ x \sim x' \iff \text{for all } s \in \Gamma \ : \ x \in s \text{ if, and only if } x' \in s.\]
Since $\Gamma$ is finite, $X/{\sim}$ is finite. We write $Y := X/{\sim}$ and $q : X \twoheadrightarrow Y$ for the quotient map. Note that $q$ is continuous with respect to the discrete topology on $Y$, since each equivalence class $[x]_{\sim}$ is clopen: it can be described by the formula $\bigwedge_{x \in \gamma \in \Gamma}\gamma \wedge \bigwedge_{x \not\in\gamma \in \Gamma} \neg \gamma$.

We will define three relations, $F$, $\leq_Y$ and $\preceq$, on the quotient $Y$.
First, define the relation $F$ by:
\[ y F y' \iff \text{there exists } x \in X \text{ such that } q(x) = y \text{ and } q(f(x)) = y',\]
that is, $F$ is the smallest filtration of the (functional) relation $f$ on $X$, cf., e.g., \cite[Lem.~2.40]{BRV2001}.
Let $\leq_Y$ be the reflexive and transitive closure of $F$, i.e.,
\begin{align*}
y \leq_Y y' \iff &\text{there exist } m \geq 0 \text{ and } z_0, \dots, z_m \in Y \text{ such that } \\
&z_0 = y, z_m = y', \text{ and } z_i F z_{i+1} \text{ for all } 0 \leq i < m.
\end{align*}

 Finally, we define a relation $\preceq$ on $Y$ by
\begin{align*} 
y \preceq y' \iff \text{for } &\text{all } x \in X, \text{ if } q(x) = y, \text{ then there exists } \\
&x' \in X \text{ such that } x \leq x' \text{ and } q(x) = y'.
\end{align*}

In the following claim, we collect several properties of the three relations defined above. In particular, (i) \& (ii) show that $F$ is a filtration of $f$, (iii) \& (iv) show that $\leq_Y$ is a filtration of $\leq_X$, (v)--(viii) provide detailed properties of the relation $\preceq$ that we need in our construction, and (ix) \& (x) show that the properties of the minimum element $x_0$ are preserved in $(Y,\preceq)$.\\

{\bf Claim 1.} For any $x, x' \in X$, $y, y' \in Y$, the following hold.
\begin{enumerate}
\item $q(x) F q(f(x))$.
\item whenever $\N t \in \Gamma$, if $q(x) F q(x')$ then $x \in \N t$ if, and only if, $x' \in t$.
\item if $x \leq_X x'$, then $q(x) \leq_Y q(x')$.
\item whenever $\Box t \in \Gamma$, if $q(x) \leq_Y q(x')$ and $x \in \Box t$, then $x' \in t$.
\item $\preceq$ is a linear preorder.
\item if $y \preceq y'$, then $y \leq_Y y'$.
\item whenever $\diam t \in \Gamma$, if $x \in \diam t$ then there exists $w$ such that $q(x) \preceq q(w)$ and $w \in t$.
\item if $y \preceq y'$ and $y' \preceq y$, then either $y = y'$, or $z \preceq y$ for all $z \in Y$.
\item $q(x) = q(x_0)$ if, and only if, $x = x_0$.
\item $q(x_0) \preceq y$ and not $y {F} q(x_0)$.
\end{enumerate}

\begin{proof}[Proof of Claim 1.]

(i) and (ii) are standard, cf. e.g. \cite[Lem. 2.40]{BRV2001}.


(iii) Note that it suffices to prove that $K_y := \{x' \in X \ | \ y \leq_Y q(x')\}$ is an up-set in $\leq_X$ for any $y \in Y$. 
Notice that $K_y$ is clopen, since $q$ is continuous and $Y$ is finite. 
To show that $K_y$ is an up-set, by Def~\ref{dfn:LTLIspace}(ii), it suffices to prove that $f(K_y) \subseteq K_y$. Indeed, if $x' \in K_y$, then $y \leq_Y q(x') {F} q(f(x'))$, and $\leq_Y$ is a transitive relation containing $F$, so $q(f(x')) \in K_y$. 

(iv) Suppose that $\Box t \in \Gamma$ and $x \in \Box t$. By induction on $m$, using that $\X\Box t \in \Gamma$ since $\Gamma$ is filterable, one can prove that if there exists an $F$-path of length $m$ from $q(x)$ to $q(x')$, then $x' \in \Box t$.
Now, if $q(x) \leq_Y q(x')$, then by definition of $\leq_Y$ there exists an $F$-path from $q(x)$ to $q(x')$, so $x' \in \Box t$. In particular, $x' \in t$.
\auxproof{We prove by induction on $m \geq 0$ that, if there exists an $F$-path of length $m$ from $q(x)$ to $q(x')$, then $x' \in \Box t$. For $m = 0$, we have $q(x) = q(x')$, and $x \in \Box t \in \Gamma$, so $x' \in \Box t$. Assume the induction hypothesis for some $m \geq 0$, and suppose that there exists an $F$-path of length $m + 1$ from $q(x)$ to $q(x')$. Let $z F q(x')$ be the last step of this $F$-path. Pick $x''$ such that $q(x'') = z$ and $q(f(x'')) = q(x')$. Since there is an $F$-path of length $m$ from $q(x)$ to $z = q(x'')$, the induction hypothesis yields that $x'' \in \Box t$. Since $\Box t \subseteq \X \Box t$, we obtain $f(x'') \in \Box t$. Now $f(x'') \sim x'$ and $\Box t \in \Gamma$ together yield $x' \in \Box t$.}

(v) Reflexivity and transitivity are straight-forward to prove. For linearity, suppose that $y \not\preceq y'$. Pick $x \in X$ such that $q(x) = y$ and for any $x'$ such that $q(x') = y'$, we have $x \nleq_X x'$. Since $\leq_X$ is linear (Prop.~\ref{prop:3dum}), we have $x' \leq_X x$ for any $x'$ such that $q(x') = y'$, so $y' \preceq y$.
\auxproof{It is clear that $\preceq$ is reflexive. If $y \preceq y' \preceq y''$, let $x \in X$ be such that $q(x) = y$. Pick $x' \geq_X x$ such that $q(x') = y'$, and $x'' \geq_X x'$ such that $q(x'') = y''$. Then, since $\leq_X$ is transitive, $x'' \geq_X x$ and $q(x'') = y''$. So $y \preceq y''$.}

(vi) Suppose that $y \preceq y'$. Pick $x \in X$ such that $q(x) = y$. By definition of $\preceq$, pick $x' \geq_X x$ such that $q(x') = y'$. By (iii), we have $y = q(x) \leq_Y q(x') = y'$.

(vii) Suppose that $\diam t \in \Gamma$ and $x \in \diam t$. Let $T := \{y \in Y \ | \ q(x) \not\preceq y\}$.
For each $y \in T$, we have $q(x) \not\preceq y$, so pick $x_y \in X$ such that $q(x_y) = q(x)$ and, for any $w$ such that $q(w) = y$, we have $x_y \nleq w$. Since $\leq_X$ is linear (Prop.~\ref{prop:3dum}), choose an enumeration $y_0, \dots, y_m$ of the elements of $T$ such that $x_{y_0} \leq_X \dots \leq_X x_{y_m}$. Since $x \in \diam t \in \Gamma$ and $x_{y_m} \sim x$, we have $x_{y_m} \in \diam t$. Pick $w \geq_X x_{y_m}$ such that $w \in t$. Note that $q(w) \not\in T$: otherwise, we would have $q(w) = y_j$ for some $y_j \in T$, and $x_{y_j} \leq_X x_{y_m} \leq_X w$, contradicting the choice of $x_{y_j}$. Therefore, since $\preceq$ is linear by (v), we must have $q(x) \preceq q(w)$, as required.

(viii) Suppose that $y \preceq y'$, $y \neq y'$ and that there exists $z \in Y$ such that $z \not\preceq y$. We prove that $y' \not\preceq y$. Let $a := \{ v \in X \ | \ q(v) \neq y\}$, a clopen subset of $X$ since $q$ is continuous. Pick an element $x \in X$ such that $q(x) = y$, i.e., $x \not\in a$. We note first that $x \in \diam \Box a$: since $z \not\preceq y$, pick $w \in X$ such that $q(w) = z$ and, whenever $w \leq_X v$, $q(v) \neq y$. Now $w \in \Box a$, so $w \not\leq_X x$, and hence $x \leq_X w$ since $\leq_X$ is linear. Now, since $A$ verifies (Dum) (Prop.~\ref{prop:3dum}), $x \not\in a$, and $x \in \diam \Box a$, we obtain $x \not\in \Box(\Box(a \to \Box a) \to a)$. Pick $x_1 \geq_X x$ such that $x_1 \in \Box(a \to \Box a)$ and $x_1 \not\in a$. By definition of $a$, we have $q(x_1) = y \preceq y'$, so we may pick $x_2 \geq_X x_1$ such that $q(x_2) = y'$. Since $y \neq y'$, we have $x_2 \in a$. Thus, since  $x_1 \in \Box(a \to \Box a)$,
and $x_1 \leq_X x_2$, we obtain $x_2 \in \Box a$. In particular, there is no $v \geq_X x_2$ such that $q(v) = y$, so $y' \not\preceq y$, as required.

(ix) If $q(x) = q(x_0)$, then since $x_0 \in \I$ and $\I \in \Gamma$, we have $x \in \I$. Therefore, $x = x_0$.

(x) By (ix) and the fact that $x_0 \leq_X x$ for any $x \in X$, it is clear that $q(x_0) \preceq y$ for any $y \in Y$. If we would have $y {F} q(x_0)$, then there would exist $x \in X$ such that $q(x) = y$ and $q(f(x)) = q(x_0)$. However, by (ix) we would get $f(x) = x_0$, contradicting Def.~\ref{dfn:LTLIspace}(iv).
\end{proof}
We will write $\equiv$ for the equivalence relation induced by $\preceq$, i.e., $y \equiv y'$ if, and only if, $y \preceq y'$ and $y' \preceq y$.
By Claim 1(viii), the preorder $(Y,\preceq)$ is a ``balloon'', i.e., there exist $\alpha \geq 0$, $\beta \geq 1$ and an enumeration $y_0, \dots, y_\alpha, \dots, y_{\alpha+\beta-1}$ of the elements of $Y$ such that, for $0 \leq i \leq \alpha - 1$, $y_i \preceq y_{i+1}$ and $y_{i+1} \not\preceq y_i$, and for $0 \leq j, j' \leq \beta-1$, $y_{\alpha+j} \equiv y_{\alpha+j'}$. In a picture:

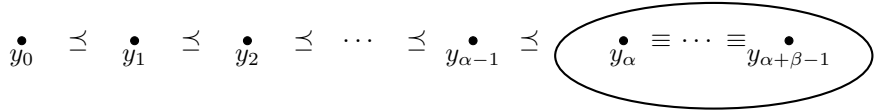
\begin{center}
\begin{tikzpicture}
\filldraw (0,0) circle (1.5 pt) node[below] {$y_0$};
\node at (0.75, 0) {$\preceq$};
\filldraw (1.5,0) circle (1.5 pt) node[below] {$y_1$};
\node at (2.25, 0) {$\preceq$};
\filldraw (3,0) circle (1.5 pt) node[below] {$y_2$};
\node at (3.75, 0) {$\preceq$};
\node at (4.5,0) {$\dots$};
\node at (5.25, 0) {$\preceq$};
\filldraw (6,0) circle (1.5 pt) node[below] {$y_{\alpha-1}$};
\node at (6.75, 0) {$\preceq$};
\draw[thick] (9.2,-.2) ellipse (60pt and 20pt);
\filldraw (8,0) circle (1.5 pt) node[below] {$y_{\alpha}$};
\node at (8.5, 0) {$\equiv$};
\node at (9,0) {$\dots$};
\node at (9.5, 0) {$\equiv$};
\filldraw (10.2,0) circle (1.5 pt) node[below] {$y_{\alpha+\beta-1}$};
\end{tikzpicture}
\end{center}
\renewcommand{\a}{\alpha}
\renewcommand{\b}{\beta}
\renewcommand{\c}{\gamma}

As a convenient notation, we extend this enumeration of the elements of $Y$ to an infinite sequence, by 
defining $y_{\a+\c} := y_{\a + (\c \;\mathrm{mod}\; \b)}$ for any $\c \geq \b$. 

We now define a function $\sigma : \omega \to Y$ and a strictly increasing sequence of natural numbers $(\ell_i)_{i\in \omega}$ with the following properties:
\begin{enumerate}
\item $\sigma(n) {F} \sigma(n+1)$ for all $n$,
\item $\sigma(\ell_i) = y_i$ for all $i$.
\end{enumerate}
Let $\sigma(0) := y_0$ and $\ell_0 := 0$. Assume that $\ell_i$ and $\sigma(n)$ have been defined correctly for all $n \leq \ell_i$. Since $y_i \preceq y_{i+1}$, by Claim 1(vi), we have $y_i \leq_Y y_{i+1}$, so we may pick $m > 0$, $\sigma(\ell_i+1), \dots, \sigma(\ell_i+m) \in Y$ with $\sigma(\ell_i+m) = y_{i+1}$ and $\sigma(\ell_i+j) {F} \sigma(\ell_i+j+1)$ for all $0 \leq j < m$. Define $\ell_{i+1} := \ell_i + m$. Note that we may indeed arrange the choice so that $m > 0$: if $y_i \neq y_{i+1}$, this is automatic, and if $y_i = y_{i+1}$ then $\beta = 1$ and $i \geq \alpha$; in this case first choose any $\sigma(\ell_i + 1) \in Y$ so that $y_i {F} \sigma(\ell_i + 1)$, and then choose an $F$-path from $\sigma(\ell_i + 1)$ to $y_i = y_{i+1}$, which can be done since $\sigma(\ell_i + 1) \preceq y_i$ so $\sigma(\ell_i + 1) \leq_Y y_i$ by Claim 1(vi).

Finally, we define the valuation $\overline{p} \in \mathcal{P}(\omega)^N$ by setting, for each $1 \leq j \leq N$,
\[ p_j := \{ n \in \omega \ | \ \sigma(n) = q(x) \text{ for some } x \in a_j\}.\]

The fact that ($\star$) in the statement of Lemma~\ref{lem:filtration} is true for this choice of $\overline{p}$ will  follow from the following claim and the fact that $\sigma$ is surjective, by property (ii) of $\sigma$.\\

{\bf Claim 2.} For any $s \in \Gamma$,
\begin{align}\tag{$P_s$}
\text{for any } x \in X, n \in \omega, \text{ if } q(x) = \sigma(n) \text{ and } x \in s^A(\overline{a}), \text{ then } n \in s^{\mathcal{P}(\omega)}(\overline{p}).
\end{align}

{\it Proof of Claim 2.} By induction on $s$, which is in negation normal form by assumption.

($s = v_j$.) If $q(x) = \sigma(n)$ and $x \in a_j$ then $n \in p_j$ by definition.

($s = \neg v_j$.) If $q(x) = \sigma(n)$ and $x \in \neg v_j$, then $x \not\in a_j$. Since $v_j \in \Gamma$, for any $x'$ with $\sigma(n) = q(x')$, we have $x' \sim x$, so $x' \not\in a_j$. Hence, $n \not\in p_j$, so $n \in \neg p_j$.

($s = \I$.) If $q(x) = \sigma(n)$ and $x \in \I$ then $x = x_0$. Then $\sigma(n) = q(x_0) = y_0$, since $q(x_0) \preceq q(x)$ for all $x$ by Claim 1(x). If we would have $n > 0$, then we would get $\sigma(n-1) F y_0$, which is impossible by Claim 1(x). So $n = 0 \in \I$.

($s = \neg \I$.) If $q(x) = \sigma(n)$ and $x \in \neg \I$ then $x \neq x_0$. By Claim 1(ix), $q(x) \neq y_0$. Thus, $\sigma(n) \neq y_0$, so in particular $n \neq 0$, so $n \in \neg I$.

($s = s_1 \vee s_2$) and ($s = s_1 \wedge s_2$) are straight-forward.

($s = \N t$.) Suppose that $q(x) = \sigma(n)$ and $x \in \N t$. Since $\sigma(n) {F} \sigma(n+1)$, pick $x'$ such that $q(x') = \sigma(n) = q(x)$ and $q(f(x')) = \sigma(n+1)$. Since $x \in \N t$, $\N t \in \Gamma$, and $q(x) = q(x')$, we have $x' \in \N t$, so $f(x') \in t$. By the induction hypothesis $(P_t)$, we get $n + 1 \in t$, so $n \in \N t = s$.

($s = \Box t$.) Suppose that $q(x) = \sigma(n)$ and $x \in \Box t$. Let $m \geq n$ be arbitrary, and pick $x'$ such that $q(x') = \sigma(m)$. It follows from property (i) of $\sigma$ that $q(x) = \sigma(n) \leq_Y \sigma(m) = q(x')$. 
By Claim 1(iv), we have $x' \in t$. Applying the induction hypothesis $(P_t)$ to $x'$ and $m$, we obtain $m \in t$. Since $m$ was arbitrary, we conclude that $n \in \Box t$.

($s = \diam t$.) Suppose that $q(x) = \sigma(n)$ and $x \in \diam t$. 
We need to prove that $n \in \diam t$, i.e., that there exists $k \geq 0$ such that $n + k \in t$.
We first prove that, {\it for any $k$, if $n + i \not\in t$ for all $0 \leq i \leq k$, then $q^{-1}(\sigma(n+k)) \subseteq \diam t$}. For $k = 0$, this is clear, because $x \in \diam t \in \Gamma$. For the induction step, if the statement holds for some $k$, assume that $q(x') = \sigma(n+k+1)$ and $n+i \not\in t$ for all $0 \leq i  \leq k + 1$. Pick some $x''$ such that $q(x'') = \sigma(n+k)$. Then $x'' \in \diam t$ by the statement for $k$, and $x'' \not\in t$, for otherwise ($P_t$) would give $n+k \in t$.
Therefore, by Def.~\ref{dfn:LTLIalgebra}(i), $x'' \in \X\diam t$, and $\X \diam t \in \Gamma$ since $\Gamma$ is filterable. By Claim 1(ii) and the fact that $q(x'') = \sigma(n+k) {F} \sigma(n+k+1) = q(x')$, we obtain $x' \in \diam t$.

Now, since $\sigma$ visits $y_\a$ infinitely often, pick $\ell \geq 0$ such that $\sigma(n+\ell) = y_\a$. If $n + i \in t$ for some $0 \leq i \leq \ell$, then we are done. Otherwise, pick some $u \in X$ such that $q(u) = y_\a$; by the previous paragraph, $u \in \diam t$. By Claim 1(vii), pick $w \in X$ such that $y_\a \preceq q(w)$ and $w \in t$. Then $q(w) = y_{\a + j}$ for some $0 \leq j \leq \b - 1$. Since $\sigma$ visits $y_{\a + j}$ infinitely often, pick $k \geq \ell$ such that $\sigma(n+k) = y_{\a + j} = q(w)$. Then $n + k \in t$ by ($P_t$).
\end{proof}

We are now in a position to prove Theorem~\ref{thm:completeness}, which, by the remarks preceding it, finishes the proof that $T^*$ is a companion of $T$.

\begin{proof}[Proof of Theorem~\ref{thm:completeness}]
Let $t(v_1,\dots,v_N)$ be any $\mathcal{L}$-term, and assume that $\mathcal{P}(\omega) \models t = \btop$. By Lemma~\ref{lem:NNF}, let $u$ be an NNF $\mathcal{L}$-term that is equivalent to $\neg t$. Let $\mathrm{Sub}(u)$ denote the set of subterms of the term $u$, and let $\Gamma$ be the filterable set
\[ \Gamma := \{s \ | \ s \in \mathrm{Sub}(u)\} \cup \{ \N \diam s \ | \ \diam s \in \mathrm{Sub}(u)\} \cup \{ \N \Box s \ | \ \Box s \in \mathrm{Sub}(u)\} \cup \{I\}.\]
Indeed, $\Gamma$ is finite, and any element of $\Gamma$ is in negation normal form.

Let $A$ be a \LTLI-algebra and let $\overline{a} \in A^N$ be arbitrary; we prove that $t^A(\overline{a}) = \btop^A$. By Theorem~\ref{thm:duality}, there exists a \LTLI-space $(X,\leq,f,x_0)$ with dual algebra (isomorphic to) $A$. Choose $\overline{p} \in \mathcal{P}(\omega)^n$ with property ($\star$) of Lemma~\ref{lem:filtration}.
Let $x \in X$ be arbitrary. Pick $n_x \in \omega$ as in Lemma~\ref{lem:filtration}. Then $n_x \in t^{\mathcal{P}(\omega)}(\overline{p})$, since $\mathcal{P}(\omega) \models t = \btop$, so $n_x \not\in (\neg t)^{\mathcal{P}(\omega)}(\overline{p}) = u^{\mathcal{P}(\omega)}(\overline{p})$. Hence, since $u \in \Gamma$, we also have $x \not\in u^A(\overline{a}) = (\neg t)^A(\overline{a})$, by the property (\ref{eq:star}) in Lemma~\ref{lem:filtration}. Therefore, $x \in t^A(\overline{A})$. Since $x$ was arbitrary, it follows that $t^A(\overline{a}) = X = \btop^A$.
\end{proof}

\section{$T^*$ is model complete}\label{sec:modelcomplete}
Let $\phi$ be a first-order formula in the signature of \LTLI-algebras. We need to show that $\phi$ is equivalent to an existential formula $\phi'$ in $T^*$, the first-order theory of $\mathcal{P}(\omega)$. We will proceed according to the following scheme.

\begin{enumerate}
\item Syntactically transform $\phi$ into a formula $\widehat{\phi}$ of the monadic second-order logic S1S (Proposition~\ref{prop:fomso}).
\item Associate to the S1S-formula $\widehat{\phi}$ a B\"uchi automaton $\mathcal{A}_{\widehat{\phi}}$ (Theorem~\ref{thm:buchi}).
\item Associate to the B\"uchi automaton $\mathcal{A}_{\widehat{\phi}}$ a \LTLI-term $t_{\mathcal{A}_{\widehat{\phi}}}$ representing it (Proposition~\ref{prop:exfo}).
\item Use an appropriate renaming of variables to obtain an existential formula $\phi'$ equivalent to $\phi$ (Subsection~\ref{subsec:proof}).
\end{enumerate}
We now perform each of the steps in Subsections~\ref{subsec:S1S}--\ref{subsec:proof} below.

\subsection{From first-order to S1S}\label{subsec:S1S}

Let us recall a definition of the syntax and semantics of the monadic second-order logic S1S (see, e.g., \cite[Ch. 12]{GTW2002} for more details and background).

\begin{definition}\label{dfn:S1S}
The set $\SF{\V}$ of \emph{S1S-formulae} with free variables contained in a set $\mathcal{V}$ is the smallest set such that:
\begin{enumerate}
\item for each $P, Q \in \V$, the formulae $P \subseteq Q$ and $S(P,Q)$ are in $\SF{\V}$,
\item if $\phi$ and $\psi$ are in $\SF{\V}$, then $\phi \vee \psi$ and $\neg \phi$ are in $\SF{\V}$,
\item if $\phi$ is in $\SF{\V}$ and $R$ is in $\mathcal{V}$ then $\exists R \phi$ is in $\SF{\V\setminus\{R\}}$.
\end{enumerate}
The \emph{satisfaction relation} ${\models} \subseteq \mathcal{P}(\omega)^\V \times \SF{\V}$ is defined inductively by
\begin{enumerate}
\item $v \models P \subseteq Q$ if, and only if, $v(P) \subseteq v(Q)$,\\
$v \models S(P,Q)$ if, and only if, there exists $n \in P$ such that $n+1 \in Q$,
\item $v \models \phi \vee \psi$ if, and only if, $v \models \phi$ or $v \models \psi$,\\
$v \models \neg \phi$ if, and only if, $v \models \phi$ does not hold,
\item $v \models \exists R \phi$ if, and only if, there exists 
a variant $v'$ of $v$, differing from $v$ only for the value of $R$, 
such that $v' \models \phi$.
\end{enumerate}
\end{definition}

Recall \cite[Def. 12.5]{GTW2002} that a quantification over an individual variable, $\exists x \phi$, can be encoded in S1S as a quantification $\exists x (\mathsf{sing}(x) \wedge \phi)$, where $\mathsf{sing}(x)$ is an S1S-formula expressing that $v(x)$ must be a singleton for any valuation $v$. We use the convention that lowercase letters $x$, $y$, $\dots$ are individual variables, while capital letters $P$, $Q$, $R$, $\dots$ are set variables. We will also make use of the standard abbreviations $x \in P$, $P = Q$, $x \leq y$, $S(x,y)$, as in \cite[Ch. 12]{GTW2002}. \\

{\bf Notation.} In order to avoid confusion, we need to distinguish the function symbols $\vee$, $\neg$, $\bot$ in the first-order signature $\mathcal{L}$ of \LTLI-algebras from the symbols $\vee$, $\neg$, $\bot$ that occur as connectives in first-order $\mathcal{L}$-formulae and in S1S-formulae. Therefore, throughout this section, we use an alternative first-order signature $\mathcal{L}^* := \{\cup,-,\emptyset,\diam,\X,\I\}$ for \LTLI-algebras, as well as abbreviations $a \cap b := -(-a \cup -b)$ and $a \Rightarrow b := -a \cup b$.\\

We now come to the translation from first-order $\mathcal{L}^*$-formulas to formulas in the logic S1S. This translation is a variant of the so-called {\it standard translation} of modal logic into monadic second-order logic, cf., e.g., \cite[Prop. 3.12]{BRV2001}.

\begin{proposition}\label{prop:fomso}
For any first-order $\mathcal{L}^*$-formula $\phi(p_1,\dots,p_n)$, there exists an S1S-formula $\widehat{\phi}(P_1,\dots,P_n)$ such that, for any $v \in \mathcal{P}(\omega)^n$,\footnote{Note that the symbol $\models$ is used with two different meanings in (\ref{eq:fomso}): on the left-hand-side, it denotes the usual satisfaction relation of first-order logic, while on the right-hand-side it denotes the satisfaction relation of S1S of Definition~\ref{dfn:S1S}. Similarly, the tuple $v \in \mathcal{P}(\omega)^n$ is regarded on the left as a valuation of the first-order variables $p_1,\dots,p_n$, and on the right as a valuation of the second-order variables $P_1,\dots,P_n$.}
\begin{equation}\label{eq:fomso}
\mathcal{P}(\omega), v \models \phi \text{ if, and only if, } v \models \widehat{\phi}.
\end{equation}
\end{proposition}
\begin{proof}
We first inductively define, for any $\mathcal{L}^*$-term $t(p_1,\dots,p_n)$, an S1S-formula $t^\bullet(P_1,\dots,P_n,x)$, where $x$ is a fresh individual variable, as follows:
\begin{itemize}
\item $(p_i)^\bullet := x \in P_i$,
\item $(t \cup u)^\bullet := t^\bullet(x) \vee u^\bullet(x)$,
\item $(-t)^\bullet := \neg (t^\bullet)(x)$,
\item $\emptyset^\bullet := \neg(x = x)$,
\item $(\diam t)^\bullet := \exists y (x \leq y \wedge t^\bullet(P_1,\dots,P_n,y))$,
\item $(\X t)^\bullet := \exists y (S(x,y) \wedge t^\bullet(P_1,\dots,P_n,y))$,
\item $\I^\bullet := \forall z (x \leq z).$
\end{itemize}
Note that, for any $i \in \omega$ and $v \in \mathcal{P}(\omega)^n$, we have
\begin{equation}\label{eq:termeq}
i \in t^{\mathcal{P}(\omega)}(v(1),\dots,v(n)) \text{ if, and only if, } v \models t^\bullet(v(1),\dots, v(n),i).
\end{equation}
Now, for any $\mathcal{L}^*$-formula $\phi(p_1,\dots,p_n)$, define $\widehat{\phi}(P_1,\dots,P_n)$ by replacing any atomic formula $t = u$ by $\forall x (t^\bullet(P_1,\dots,P_n,x) \leftrightarrow u^\bullet(P_1,\dots,P_n,x))$, and any occurrence of $\exists p$ by $\exists P$. The claimed equivalence is proved by an easy induction on the complexity of $\phi$, using (\ref{eq:termeq}) for the base case.
\end{proof}

\subsection{From S1S-formula to B\"uchi automaton}\label{subsec:automata}
We briefly recall the definition of finite automata and the B\"uchi acceptance condition. See, e.g., \cite[Ch.~1]{GTW2002} for more details.
\begin{definition}\label{dfn:buchi}
Let $\Sigma$ be a finite alphabet. A (finite, non-deterministic) \emph{automaton} is a tuple $\mathcal{A} = (Q,\delta,q_0,F)$ where $Q$ is a finite set, whose elements are called \emph{states}, $q_0 \in Q$ is a distinguished element called \emph{the initial state}, $F \subseteq Q$ is a set, whose elements are called \emph{final states}, and $\delta \subseteq Q \times \Sigma \times Q$ is a ternary relation, called the \emph{transition relation}. For an infinite word $w \in \Sigma^\omega$, a function $\rho \colon \omega \to Q$ is called a \emph{run of $\mathcal{A}$ on $w$} if $\rho(0) = q_0$ and, for any $i \in \omega$, $(\rho(i),w(i),\rho(i+1)) \in \delta$. For any word $w \in \Sigma^\omega$, the automaton $\mathcal{A}$ \emph{B\"uchi-accepts} $w$ if, and only if, there exist a run $\rho$ of $\mathcal{A}$ on $w$ and a final state $q_f \in F$ such that the set $\{ i \in \omega \ | \ \rho(i) = q_f\}$ is infinite.
\end{definition}

We next recall how valuations of variables naturally define infinite words.

\begin{definition}\label{dfn:wv}
Let $n \geq 0$ and $\Sigma := \mathcal{P}(\{1,\dots,n\})$. For any $v \in \mathcal{P}(\omega)^n$, we define the infinite word $w_v \in \Sigma^\omega$ by
\[ w_v(i) := v^{-1}(\{ u \in \Sigma \ | \ i \in u\}) = \{k \ | \ i \in v(k) \}, \quad i \in \omega.\]
\end{definition}

The following theorem, originally due to B\"uchi \cite{Buc1962}, states that any S1S-formula can be converted into an automaton.
\begin{theorem}\label{thm:buchi}
For any S1S-formula $\psi(P_1,\dots,P_n)$, there exists an automaton $\mathcal{A}_\psi$ on the finite alphabet $\Sigma := \mathcal{P}(\{1,\dots,n\})$ such that, for any $v \in \mathcal{P}(\omega)^n$,
\begin{equation} \label{eq:buchi}
v \models \psi  \text{ if, and only if, } \mathcal{A}_\psi \text{ B\"uchi-accepts } w_v,
\end{equation}
\end{theorem}
\begin{proof}
See, e.g., \cite[Thm. 12.15]{GTW2002} or \cite[Thm. 5.9]{Tho1996}. 
\end{proof}

\subsection{From automata to existential formulae}\label{subsec:exform}
It is well-known that a B\"uchi automaton can be transformed into a formula of S1S which starts with a block of existential monadic set quantifiers, but also has some quantifications over individual variables after that. As we remarked in the introduction to the paper, we require slightly more, namely that the quantifications over individual variables are not needed in the signature $\mathcal{L}^*$.

In order to state precisely the translation back from a B\"uchi automaton to an existential $\mathcal{L}^*$-formula, we need to recall how infinite words in $\Sigma^\omega$ yield valuations of $\Sigma$ in $\mathcal{P}(\omega)$, by the reverse process to Definition~\ref{dfn:wv}.
\begin{definition}\label{dfn:vw}
Let $\Sigma$ be a finite alphabet. For any $w \in \Sigma^\omega$, we define the valuation $v_w : \Sigma \to \mathcal{P}(\omega)$ by
\[v_w(a) := w^{-1}(a) = \{i \in \omega : w(i) = a\}, \quad a \in \Sigma.\]
\end{definition}

Let $\mathcal{A} = (Q,\delta,q_0,F)$ be an automaton on a finite alphabet $\Sigma$. We now define a number of $\mathcal{L}^*$-terms with variables from the set $Q \cup \Sigma$. First define the $\mathcal{L}^*$-terms:
\begin{itemize}
\item $\mathsf{Init} := \I \Rightarrow q_0$,
\item $\mathsf{Trans} := \bigcap_{q \in Q} \left( q \Rightarrow \bigcup_{(q,a,q') \in \delta} (a \cap \X q') \right)$,
\item $\mathsf{Part} := \bigcup_{q \in Q} (q \cap \bigcap_{\substack{q' \in Q\\ q' \neq q}} -q')$,
\item $\mathsf{Accept} := \bigcup_{q \in F} \diam q$,
\end{itemize}
and define the $\mathcal{L}^*$-term $t_{\mathcal{A}}$ by:
\begin{equation}\label{eq:termdef}
t_\mathcal{A}(\overline{a}, \overline{q}) :=  \mathsf{Part} \cap \mathsf{Init} \cap \mathsf{Trans}  \cap \mathsf{Accept}.
\end{equation}

\begin{proposition}\label{prop:exfo}
Let $\mathcal{A}$ be an automaton on a finite alphabet $\Sigma$. For any $w \in \Sigma^\omega$,
\begin{equation}\label{eq:exfo}
 \mathcal{A} \text{ B\"uchi-accepts } w \text{ if, and only if, } \mathcal{P}(\omega), v_w \models \exists \overline{q} \, (t_{\mathcal{A}} = \top).
\end{equation}
\end{proposition}
\begin{proof}
Clear from the definitions.
\end{proof}

\subsection{Proof that $T^*$ is model complete}\label{subsec:proof}
Let $\phi(p_1,\dots,p_n)$ be a first-order $\mathcal{L}^*$-formula. Let $\widehat{\phi}(P_1,\dots,P_n)$ be the S1S-formula given by Proposition~\ref{prop:fomso}. Let $\mathcal{A}_{\widehat{\phi}}$ be the automaton on the alphabet $\Sigma := \mathcal{P}(1,\dots,n) = \{a_1,\dots,a_{2^n}\}$ given by Theorem~\ref{thm:buchi}. Let $t_{\mathcal{A}_{\widehat{\phi}}}$ be the $\mathcal{L}^*$-term defined in (\ref{eq:termdef}). We define the existential $\mathcal{L}^*$-formula
\[ \phi'(p_1,\dots,p_n) :=
\exists \overline{a} \, \exists \overline{q} \left(
(t_{\mathcal{A}_{\widehat{\phi}}} = \top) \wedge
\bigwedge_{\ell=1}^{2^n} \left(a_\ell = \bigcap_{k \in a_\ell} p_k \cap \bigcap_{k \not\in a_\ell} -p_k \right) 
 \right).\]
\\

{\bf Claim.} $T^* \proves \phi \leftrightarrow \phi'$.

\begin{proof}
It suffices to prove that $\mathcal{P}(\omega), v \models \phi$ if, and only if, $\mathcal{P}(\omega), v \models \phi'$, for any $v \in \mathcal{P}(\omega)^n$.
Let $v \in \mathcal{P}(\omega)^n$ be arbitrary. We obtain:
\begin{align*}
\mathcal{P}(\omega), v \models \phi &\iff v \models \widehat{\phi} \quad &\text{(by (\ref{eq:fomso}) in Prop.~\ref{prop:fomso})}\\
&\iff \mathcal{A}_{\widehat{\phi}} \text{ accepts } w_v \quad &\text{(by (\ref{eq:buchi}) in Thm.~\ref{thm:buchi})}\\
&\iff \mathcal{P}(\omega), v_{w_v} \models \exists \overline{q} \, (t_{\mathcal{A}_{\widehat{\phi}}} = \top)\quad &\text{(by (\ref{eq:exfo}) in Prop.~\ref{prop:exfo})}\\
&\iff \mathcal{P}(\omega), v \models \phi',
\end{align*}
where the last equivalence follows from the fact that, for each $a \in \Sigma = \mathcal{P}(1,\dots,n)$, we have
\[ v_{w_v}(a) = \{i \in \omega : a = \{k : i \in v(k)\}\} = \bigcap_{k \in a} v(k) \cap \bigcap_{k \not\in a} -v(k).\]
\end{proof}

\section{Conclusions and future work}\label{sec:conclusions}
In this paper, we characterized monadic second order logic on infinite words as the model companion of the universal class of \LTLI-algebras. This is not a stand-alone result; we indicate a few of the possible directions of further research here.

First, we expect that it is possible to give a similar characterization of monadic second order logic on finite words, by adjusting the axiomatization of \LTLI-algebras in the appropriate manner. A harder, but equally natural question, is whether a version of our result holds for monadic second order logics on tree structures, such as S2S, or finite trees.

Our result (and/or its incarnation for finite words) is likely to be related to the duality theory for regular languages and logics that is being developed in a series of papers including \cite{GGP2008,GKP2014}. It would be interesting to make these connections explicit, in order to determine if the two approaches can benefit from each other.

Riba \cite{Rib2012} gives a model-theoretic proof of the completeness of Siefkes' axiomatization of S1S. Our result in this paper entails that S1S, viewed as a first-order theory, coincides with the theory of the existentially closed \LTLI-algebras \cite[Thm.~B]{Whe1976}. Therefore, we suggest that an alternative axiomatization of S1S, and completeness proof for it, could be sought by axiomatizing  existentially closed \LTLI-algebras. In a similar direction, we note that Gheerbrant and Ten Cate \cite{GheCat2012} used modal logic techniques to axiomatize monadic second order logic on finite trees, and fragments. An extension of our results in this paper to finite trees could also be connected to the results in \cite{GheCat2012}.

\subsection*{Acknowledgement}
We thank the referee for their valuable comments on an earlier version of this paper.


\bibliographystyle{asl}
\bibliography{ghilardivangool2015-revised}

\end{document}